\documentclass[11pt]{amsart} \textwidth=14.5cm \oddsidemargin=1cm
\usepackage{amsmath}
\usepackage{amsxtra}
\usepackage{amscd}
\usepackage{amsthm}
\usepackage{amsfonts}
\usepackage{amssymb}
\usepackage{eucal}
\usepackage[leqno]{easyeqn}

\newtheorem{thm}{Theorem}[section]
\newtheorem{lem}{Lemma}[section]
\newtheorem{cor}{Corollary}[section]
\newtheorem{prop}{Proposition}[section]

\theoremstyle{definition}
\newtheorem{definition}{Definition}
\newtheorem{example}{Example}[section]

\theoremstyle{remark}
\newtheorem{rem}{Remark}[section]

\numberwithin{equation}{section}

\begin{document}

%Referring commands:
\newcommand{\thmref}[1]{Theorem~\ref{#1}}
\newcommand{\secref}[1]{Section~\ref{#1}}
\newcommand{\lemref}[1]{Lemma~\ref{#1}}
\newcommand{\propref}[1]{Proposition~\ref{#1}}
\newcommand{\corref}[1]{Corollary~\ref{#1}}
\newcommand{\remref}[1]{Remark~\ref{#1}}
\newcommand{\eqnref}[1]{(\ref{#1})}
\newcommand{\exref}[1]{Example~\ref{#1}}

%Simplified symbols:
\newcommand{\nc}{\newcommand}
\nc{\ZZ}{{\mathbb Z}} \nc{\C}{{\mathbb C}} \nc{\N}{{\mathbb N}}
\nc{\F}{{\mf F}} \nc{\Q}{\ol{Q}} \nc{\la}{\lambda}
\nc{\ep}{\epsilon} \nc{\h}{\mathfrak h}
\nc{\sh}{\overline{\mathfrak h}} \nc{\n}{\mf n} \nc{\A}{{\mf a}}
\nc{\G}{{\mathfrak g}} \nc{\SG}{\overline{\mathfrak g}}
\nc{\D}{\mc D} \nc{\Li}{{\mc L}} \nc{\La}{\Lambda}
\nc{\is}{{\mathbf i}} \nc{\V}{\mf V} \nc{\bi}{\bibitem}
\nc{\NS}{\mf N} \nc{\dt}{\mathord{\hbox{${\frac{d}{d t}}$}}}
\nc{\E}{\mc E} \nc{\ba}{\tilde{\pa}} \nc{\half}{\frac{1}{2}}
\nc{\mc}{\mathcal} \nc{\mf}{\mathfrak} \nc{\hf}{\frac{1}{2}}
\nc{\hgl}{\widehat{\mathfrak{gl}}} \nc{\gl}{{\mathfrak{gl}}}
\nc{\hz}{\hf+\ZZ} \nc{\vac}{|0 \rangle}
\nc{\dinfty}{{\infty\vert\infty}} \nc{\SLa}{\overline{\Lambda}}
\nc{\SF}{\overline{\mathfrak F}} \nc{\SP}{\overline{\mathcal P}}
\nc{\U}{\mathfrak u} \nc{\SU}{\overline{\mathfrak u}}
\nc{\ov}{\overline}

\advance\headheight by 2pt

\title[A BGG-type resolution for tensor $\gl(m|n)$-modules]{A BGG-type resolution for tensor modules over
general linear superalgebra}

\author[Shun-Jen Cheng]{Shun-Jen Cheng$^\dagger$}
\thanks{$^\dagger$Partially supported by an NSC-grant of the ROC and an Academia Sinica Investigator grant}
\address{Institute of Mathematics, Academia Sinica, Taipei,
Taiwan 11529} \email{chengsj@math.sinica.edu.tw}

\author[Jae-Hoon Kwon]{Jae-Hoon Kwon$^{\dagger\dagger}$}
\thanks{$^{\dagger\dagger}$Partially supported by KRF-grant 2005-070-C00004}
\address{Department of Mathematics, University of Seoul, 90, Cheonnong-Dong, Dongdaemun-gu, Seoul 130-743, Korea}
\email{jhkwon@uos.ac.kr}

\author[Ngau Lam]{Ngau Lam$^{\dagger\dagger\dagger}$}
\thanks{$^{\dagger\dagger\dagger}$Partially supported by an NSC-grant 96-2115-M-006-008-MY3 of the ROC}
\address{Department of Mathematics, National Cheng-Kung University, Tainan, Taiwan 70101}
\email{nlam@mail.ncku.edu.tw}

\begin{abstract} \vspace{.3cm}
We construct a Bernstein-Gelfand-Gelfand type resolution in terms
of direct sums of Kac modules for the finite-dimensional
irreducible tensor representations of the general linear
superalgebra. As a consequence it follows that the unique maximal
submodule of a corresponding reducible Kac module is generated by
its proper singular vector.

 \vspace{.3cm} \noindent{\bf Key words:}
Bernstein-Gelfand-Gelfand resolution, singular vector, Kac module,
general linear superalgebra.

\vspace{.3cm}
 \noindent{\bf Mathematics Subject Classifications (2000)}: 17B67.
\end{abstract}
%\noindent{\bf Key words:} Lie superalgebra, differential
%operators, free field realization, Howe duality.

\maketitle

%\tableofcontents

\section{Introduction}

The classical result of Bernstein-Gelfand-Gelfand \cite{BGG}
resolves a finite-dimensional irreducible module over a
finite-dimensional semi-simple Lie algebra in terms of direct sums
of Verma modules. Such a resolution is sometimes called a strong BGG
resolution.  In \cite{Le, RC} it was shown that the
finite-dimensional simple modules may also be resolved in terms of
direct sums of generalized Verma modules.

While BGG resolutions have been known to exist for integrable
representations over Kac-Moody algebras (see e.g.~\cite{RCW, Ku}),
virtually nothing is known even for finite-dimensional simple Lie
superalgebras. However, what seems to be known to experts is that,
in general, the finite-dimensional simple modules over a
finite-dimensional simple Lie superalgebra cannot be resolved in
terms of Verma modules. For example, even the natural representation
of the Lie superalgebra $\mf{sl}(1|2)$ (or $\mf{gl}(1|2)$) cannot
have a resolution in terms of Verma modules (see \exref{example}).

It is therefore surprising that resolutions for a large class of
finite-dimensional representations of the general linear
superalgebra $\gl(m|n)$ in terms of Kac modules do exist. The
purpose of this article is to construct such a resolution for every
irreducible tensor module (see \secref{tensormodule}) of $\gl(m|n)$.

Roughly the idea of the construction is to exploit the connection
between the irreducible tensor representations of the Lie
superalgebra $\gl(m|n)$ and the polynomial representations of the
general linear algebra $\gl(m+n)$ in the limit $n\to\infty$. This
allows us to construct a ``weak'' resolution. The strong
resolution is then obtained from the weak version using Brundan's
Kazhdan-Lusztig theory of $\gl(m|n)$ \cite{B}.

All vector spaces, algebras and tensor products are over the complex
number field $\C$.

\section{Preliminaries}

Let $m\in\N$ and $n\in\N\cup\{\infty\}$, and set
$I(m|n)=\{-m,\cdots,-1,1,\dots,n\}$ for $n \in{\mathbb N}$, and
$I(m|n)=\{-m,\cdots,-1 \}\cup{\mathbb N}$ for $n=\infty$. Let
$\mathcal P_{m|n}$ denote the set of partitions
$\la=(\la_{-m},\cdots,\la_{-1},\la_1,\la_2,\cdots)$ with $\la_1\le
n$.  The set $\mathcal P_{m|\infty}$ is the set of all partitions.
For a partition $\la$, we use $\lambda'$, $\ell(\lambda)$, and
$|\lambda|$ to denote its conjugate, length, and size,
respectively.

\subsection{The Lie algebra $\gl(m+n)$} We let $\C^{m+n}$ stand for the complex space of dimension
$m+n$ with the standard basis $\{\,{e}_i\,|\,i\in I(m|n)\,\}$. Let
$\G=\gl(m+n)$ be the general linear algebra which acts naturally on
$\C^{m+n}$. In the case of $n=\infty$, we let $\G$ consist of linear
transformations vanishing on all but finitely many $e_j$'s. Denote
by $\{\,{E}_{ij}\,|\,i,j\in I(m|n)\,\}$ the set of elementary
matrices in $\G$. Then $\{\,{E}_{jj}\,\vert\, j\in I(m|n)\,\}$ spans
the standard Cartan subalgebra $\h=\h_n$, while
$\{\,{E}_{ij}\,\vert\, i\le j\,\}$ spans the standard Borel
subalgebra. For $\la\in\h^*$ we denote by $L(\G,\la)$ the
irreducible highest weight $\G$-module with highest weight $\la$.

Let $\epsilon_j\in\h^*$ be determined by
$\langle\epsilon_j,E_{ii}\rangle=\delta_{ij}$ for $i,j\in I(m|n)$.
Let $\alpha_i=\epsilon_i-\epsilon_{i+1}$, for $i\in I(m|n)$ such
that $i+1\in I(m|n)$, and $\alpha_{-1}=\epsilon_{-1}-\epsilon_1$.
Then the set $\{\alpha_i\}$ is a set of simple roots of
$\G'=[\G,\G]$, and we denote the set of positive and negative
roots by $\Delta^\pm$, respectively.  Let
$\Delta_0^\pm=\Delta^\pm\cap\left(\sum_{i\not=-1}\ZZ\alpha_i\right)$
and $\Delta^\pm(0)=\Delta^\pm\setminus\Delta^\pm_0$.

Let $\{{\alpha}^{\vee}_i\}$ denote the corresponding simple coroots
and let $\{e_i,f_i,{\alpha}^{\vee}_i\}$ be the corresponding
Chevalley generators of $\G'$. Let $\rho_c\in\h^*$ be determined by
$\langle\rho_c,E_{jj}\rangle=-j$ for $j<0$, and
$\langle\rho_c,E_{jj}\rangle=1-j$ for $j>0$.

The Lie algebra $\G$ has a $\ZZ$-gradation determined by the
eigenvalues of the operator
$\hf\left(\sum_{i<0}E_{ii}-\sum_{j>0}E_{jj}\right)$.  We have
\begin{equation*}
\G=\G_{-1}\oplus\G_0\oplus\G_{+1}.
\end{equation*}
Note that $\G_0\cong\gl(m)\oplus\gl(n)$ and
$\G_{-1}\cong\C^{m*}\otimes\C^n$ as $\G_0$-modules. Set ${\mf
p}:=\G_0\oplus\G_{+1}$ and let $L^0(\la)$ be the irreducible
representation of $\G_0$ with highest weight $\la\in\h^*$. We extend
$L^0(\la)$ trivially to a ${\mf p}$-module, for which we also write
$L^0(\la)$. Denote the generalized Verma module by
$$V(\G,\la):=\text{Ind}_{{\mf p}}^{\G}L^0(\la).$$

\subsection{The Lie superalgebra $\gl(m|n)$}
Now we let $\C^{m|n}$ stand for the complex superspace of dimension
$(m|n)$ with the standard basis $\{\,\overline{e}_i\,|\,i\in
I(m|n)\,\}$. We assume that ${\rm deg}\overline{e}_i =0$ and $1$ if
$i<0$ and $i>0$, respectively. Let $\SG=\gl(m|n)$ be the general
linear superalgebra acting naturally on $\C^{m|n}$. For $n=\infty$,
we use a similar convention as before. Denote by
$\{\,\overline{E}_{ij}\,|\,i,j\in I(m|n)\,\}$ the set of elementary
matrices in $\SG$. Then $\{\,\overline{E}_{jj}\,\vert\, j\in
I(m|n)\,\}$ spans the standard Cartan subalgebra $\sh=\sh_n$, while
$\{\,\overline{E}_{ij}\,\vert\, i\le j\,\}$ spans the standard Borel
subalgebra $\ov{\mf{b}}$. For $\la\in\sh^*$, we denote by
$L(\SG,\la)$ the irreducible highest weight $\SG$-module with
highest weight $\la$.

Let $\delta_j\in\sh^*$ be determined by
$\langle\delta_j,\ov{E}_{ii}\rangle=\delta_{ij}$ and let
$\rho_s\in\sh^*$ be determined by
$\langle\rho_s,\ov{E}_{jj}\rangle=-j$ for $i,j\in I(m|n)$. Let
$\beta_i=\delta_i-\delta_{i+1}$ for $i\in I(m|n)$ such that $i+1\in
I(m|n)$, and $\beta_{-1}=\delta_{-1}-\delta_1$. Then
$\{\,\beta_i\,\}$ is a set of simple roots of $\SG'=[\SG,\SG]$.

The Lie superalgebra $\SG$ also has a $\ZZ$-gradation determined
by the eigenvalues of the of the operator
$\hf\left(\sum_{i<0}\ov{E}_{ii}-\sum_{j>0}\ov{E}_{jj}\right)$. We
have
\begin{equation*}
\SG=\SG_{-1}\oplus\SG_0\oplus\SG_{+1}.
\end{equation*}
Note that $\SG_0\cong\G_0$ and $\SG_{-1}\cong\C^{m*}\otimes\C^n$ as
$\SG_0$-modules. We set $\ov{\mf p}:=\SG_0\oplus\SG_{+1}$.  Given
$\la\in\sh^*$, we may extend $L^0(\la)$ trivially to a $\ov{\mf
p}$-module, which we also denote by $L^0(\la)$. Define the {\em Kac
module} to be
$$V(\SG,\la):=\text{Ind}_{\ov{\mf p}}^{\SG}L^0(\la).$$
\begin{definition} A $\SG$-module $V$ is said to have a {\em Kac flag} if it has a
filtration of $\SG$-modules of the form
\begin{equation*}
0=V_0\subseteq V_1\subseteq V_2\subseteq \cdots\subseteq
V_{l-1}\subseteq V_l=V,
\end{equation*}
such that $V_j/V_{j-1}$ is isomorphic to a Kac module for
$j=1,\ldots,l$.
\end{definition}

\begin{definition} Let $n\in\N\cup\{\infty\}$. Given a sequence of
integers of the form
\begin{EQA}[c]\yesnumber\label(~){aux100}
\mu=(\mu_{-m},\cdots,\mu_{-1},\mu_1,\mu_2,\cdots),\label(~.a){aux100-a}\\
\noalign{\hbox{with $\mu_k=0$ for $k\gg 0$ when $n=\infty$, and}}\\
\mu=(\mu_{-m},\cdots,\mu_{-1},\mu_1,\mu_2,\cdots,\mu_n),\label(~.b){aux100-b}
\end{EQA}
when $n\in\N$, we may interpret it as $\sum_{i\ge
-m,i\not=0}\la_i\epsilon_i\in\h^*_n$ or $\sum_{i\ge
-m,i\not=0}\la_i\delta_i\in\sh^*_n$. Suppose now that $\mu$ as in
\eqnref{aux100} such that $(\mu_1,\mu_2,\cdots)$ is a partition. We
define $\mu^\natural$ to be the integer sequence
\begin{equation}\label{def:natural}
\mu^\natural:=(\mu_{-m},\cdots,\mu_{-1},\mu_1',\mu_2',\cdots).
\end{equation}

Let $\widetilde{\mc{X}}_{m|n}$ be the set of integer sequences of
the form \eqnref{aux100} with $\mu_j\ge\mu_{j+1}$, for all $j< n$
with $j\not= 0,\, -1$. Let
$\mc{X}_{m|n}\subseteq\widetilde{\mc{X}}_{m|n}$ consist of those
$\mu$'s such that $(\mu_1,\mu_2,\cdots)$ is a partition. For
$\mu\in\mc{X}_{m|n}$, $\mu^\natural$ is well-defined, and the map
$\mu\rightarrow\mu^\natural$ is a bijection on $\mc{X}_{m|\infty}$.
\end{definition}

\subsection{Irreducible tensor
$\gl(m|n)$-modules}\label{tensormodule} The tensor powers of
$\C^{m|n}$ are completely reducible as  $\SG$-modules. Indeed the
irreducible representations that appear in these decompositions are
as follows.  An irreducible representation of $\SG$ appears as a
component of $\left(\C^{m|n}\right)^{\otimes k}$ if and only if it
is of the form $L(\SG,\la^\natural)$, where $\la\in\mc P_{m|n}$ with
$|\la|=k$ \cite{S, BR}. We call these irreducible $\SG$-modules {\em
irreducible tensor $\SG$-modules}.

Let $\la\in\mc{P}_{m|\infty}$. Clearly as $\G_0$-modules $L(\G,\la)$
and $L(\SG,\la^\natural)$ are direct sums of $L^0(\eta)$ with
$\eta\in\mc{X}_{m|\infty}$. We have the following description of
irreducible tensor $\SG$-modules.
\begin{prop}\label{hookschur}
Assume that $n=\infty$. For $\la\in\mc{P}_{m|\infty}$ and $\eta\in
\mc{X}_{m|\infty}$, the $\G_0$-module $L^0(\eta)$ is an irreducible
component of $L(\G,\la)$ if and only if the $\SG_0$-module
$L^0(\eta^\natural)$ is an irreducible component of
$L(\SG,\la^\natural)$.  Furthermore, their multiplicities coincide.
\end{prop}

\begin{proof}
This is an immediate consequence of the well-known fact that the
character of $L(\SG,\la^\natural)$ is given by the so-called Hook
Schur function associated with $\la^\natural$ \cite[Theorem
6.10]{BR}.
\end{proof}

\begin{rem}\label{decomposition} For a partition $\la$ with $\ell(\la)\le m+n$ and $k\ge 0$, it is
easy to see that $\Lambda^k(\G_{-1})\otimes {L}(\G,\la)$ as a
$\G_0$-module decomposes into a direct sum of irreducible
$\G_0$-modules with highest weights belonging to $\mc{X}_{m|n}$.
Similarly, for $\la\in\mc{P}_{m|n}$ and $k\ge 0$,
$\Lambda^k(\SG_{-1})\otimes {L}(\SG,\la^\natural)$ as a
$\SG_0$-module decomposes into a direct sum of irreducible
$\SG_0$-modules of the form $L^0(\mu)$ with $\mu\in\mc{X}_{m|n}$.
\end{rem}

\section{Eigenvalues of Casimir operators}

Throughout this section, we assume that $n=\infty$ unless otherwise
specified.

We fix a symmetric bilinear form $(\cdot\vert\cdot)_c$ on $\h^*$
satisfying
\begin{align}\label{bilineardef}
&(\la\vert \epsilon_i)_c=\langle \la,E_{ii}\rangle,  \quad
\la\in\h^*,i\in I(m|n).
\end{align}
By defining $({\alpha}^{\vee}_i\vert
{\alpha}^{\vee}_j)_c:=(\alpha_i\vert\alpha_j)_c$ for simple
coroots ${\alpha}^{\vee}_i$ and ${\alpha}^{\vee}_j$, we obtain a
symmetric bilinear form on the Cartan subalgebra of $\G'$, which
can be extended to a non-degenerate invariant symmetric bilinear
form on $\G'$ such that
\begin{equation}\label{aux:casimir2}
(e_i\vert f_j)_c=\delta_{ij}.
\end{equation}
Since every root space $\G_\alpha$ is one-dimensional, we can choose
a basis $\{u_{\alpha}\}$ of $\G_\alpha$ for $\alpha\in\Delta^+$ and
a dual basis $\{u^{\alpha}\}$ of $\G_{-\alpha}$ with respect to
$(\cdot\vert\cdot)_c$.

Let $V=\bigoplus_{\mu} V_\mu$ be a highest weight $\G$-module, where
$V_\mu$ denotes the $\mu$-weight space of $V$. Define
$\Gamma_1:V\rightarrow V$ to be the linear map that acts as the
scalar $(\mu+2\rho_c\vert\mu)_c$ on $V_\mu$. Let
$\Gamma_2:=2\sum_{\alpha\in\Delta^+}u^{\alpha}u_{\alpha}$.  The {\em
Casimir operator} (cf.~\cite{J}) is defined to be
\begin{equation*}
\Omega:=\Gamma_1+\Gamma_2.
\end{equation*}
It follows from \eqnref{bilineardef} and \eqnref{aux:casimir2} that
$\Omega$ commutes with the action of $\G$ on $V$
(cf.~\cite[Proposition 3.6]{J}). Thus, if $V$ is generated by a
highest weight vector with highest weight $\la$, then $\Omega$ acts
on $V$ as the scalar $(\la+2\rho_c\vert\la)_c$.

To produce the Casimir operator for $\SG$ we fix a symmetric
bilinear form $(\cdot\vert\cdot)_s$ on $\sh^*$ satisfying
\begin{align*}
&(\la\vert \delta_i)_s=-\text{sign}(i)\langle \la,E_{ii}\rangle,
\quad \la\in\sh^*,i\in I(m|n).
\end{align*}
An analogous argument allows us to generalize the construction above
and define the Casimir operator $\ov{\Omega}$ of the Lie
superalgebra $\SG$ that acts on a highest weight module with highest
weight $\gamma\in\overline{\h}^*$ as the scalar
$(\gamma+2\rho_s|\gamma)_s$.  We omit the details.

We will need the Weyl group of $\gl(m+\infty)$ in the sequel. For
each $\alpha_j$, define simple reflection $\sigma_j$ by
\begin{equation*}
\sigma_j(\mu):=\mu-\langle\mu,{\alpha}^{\vee}_j\rangle\alpha_j,
\end{equation*}
where $\mu\in\h^*$. Let $W$ be the subgroup of ${\rm Aut}(\h^*)$
generated by the $\sigma_j$'s. For each $w\in W$, we let $l(w)$
denote the length of $w$. We have an action on $\h$ given by
$\sigma_j(h)=h-\langle\alpha_j,h\rangle {\alpha}^{\vee}_j$ for
$h\in\h$, so that $\langle w(\mu),w(h)\rangle=\langle\mu,h\rangle$
for $\mu\in\h^*$ and $h\in\h$. We also define
\begin{equation*}
w\circ\mu:=w(\mu+\rho_c)-\rho_c, \quad w\in W,\ \mu\in\h^*.
\end{equation*}
Consider $W_{0}$ the subgroup of $W$ generated by $\sigma_j$ with
$j\not=-1$. Let
\begin{equation*}
W^{0}:=\{\,w\in W\,\vert\,
w(\Delta^-)\cap\Delta^+\subseteq\Delta^+(0)\,\}.
\end{equation*}
It is well-known that $W=W_0\, W^0$ and $W^0$ is the set of the
minimal length representatives of the right coset space
$W_0\backslash W$ (cf.~\cite[1.3.17]{Ku}). For $k\in\ZZ_+$, set
$$W^0_k:=\{\,w\in W^0\,\vert\, l(w)=k\,\}.$$

Given $\lambda\in \mc{P}_{m|\infty}$, we have
$\langle\la,{\alpha}^{\vee}_j\rangle\in\ZZ_+$ for all $j$. Since
$w\in W^0$ implies that $w^{-1}(\Delta^+_0)\subseteq\Delta^+$, we
obtain $\langle w\circ \la,{\alpha}^{\vee}_j\rangle\in\ZZ_+$, for
all $j\not=-1$, and $w\circ\lambda\in\mc{X}_{m|\infty}$.

The following proposition is well-known from the theory of standard
modules over generalized Kac-Moody algebras (see
e.g.~\cite[Proposition 3.11]{J}).

\begin{prop}\label{aux01} For
$\la\in\mc{P}_{m|\infty}$ and $\eta\in \mc{X}_{m|\infty}$, the
irreducible ${\mathfrak g}_0$-module $L^0(\eta)$ is a component of
$\Lambda^k({\mathfrak g}_{-1})\otimes {L}(\G,\la)$ with $(\eta
+2\rho_c\vert\eta)_c=(\la+2\rho_c\vert\la)_c$ if and only if there
exists $w\in W^0_k$ with $w\circ\la=\eta$. Furthermore each such
$L^0(\eta)$ appears with multiplicity one.
\end{prop}

\begin{lem}\label{mult:same}  For
$\la\in\mc{P}_{m|\infty}$ and $\eta\in \mc{X}_{m|\infty}$,
$L^0(\eta)$ is an irreducible ${\mathfrak g}_0$-module in
$\Lambda^k({\mathfrak g}_{-1})\otimes {L}(\G,\la)$ if and only if
$L^0(\eta^\natural)$ is an irreducible $\SG_0$-module in
$\Lambda^k(\SG_{-1})\otimes {L}(\SG,\la^\natural)$. Furthermore, the
multiplicities are the same.
\end{lem}

\begin{proof}
The symmetric \cite[Theorem 2.1.2]{H} and skew-symmetric
\cite[Theorem 4.1.4]{H} $(\gl,\gl)$-Howe dualities give the precise
decompositions of $\Lambda^k(\SG_{-1})\cong S^k(\C^{m*}\otimes\C^n)$
and $\Lambda^k(\G_{-1})\cong \Lambda^k(\C^{m*}\otimes\C^n)$ as
$\G_0$-modules, respectively.  From these decompositions one sees
that $L^0(\eta)$ is an irreducible component in $\Lambda^k(\G_{-1})$
if and only if $L^0(\eta^\natural)$ is an irreducible component in
$\Lambda^k(\SG_{-1})$. This fact combined with \propref{hookschur}
and the compatibility of $\natural$ under tensor products completes
the proof.
\end{proof}

We need the following combinatorial lemma.

\begin{lem}\label{aux111} Let $\la=(\la_1, \la_2,\cdots,\la_N)$ be a partition with $\ell(\la)\le
N$. For $1\le i\le N$ the sets $\{\la'_i-i+\hf\,\vert\,
\la'_i-i+\hf>0\}$ and $\{-\la_i+i-\hf \,\vert\, \la_i-i+\hf<0\}$ are
disjoint. Moreover, $\{\la'_i-i+\hf\,\vert\,
\la'_i-i+\hf>0\}\cup\{-\la_i+i-\hf \,\vert\, \la_i-i+\hf<0\}$ is a
permutation of the set $\{\hf,\frac{3}{2},\cdots,N-\hf\}$.
\end{lem}

\begin{proof}
The sets  $\{\la'_i-i+\hf\,\vert\, \la'_i-i+\hf>0\}$,
$\{-\la_i+i-\hf \,\vert\, \la_i-i+\hf<0\}$,
$\{\hf,\frac{3}{2},\cdots,N-\hf\}$ are denoted by $A$, $B$ and $C$,
respectively. We first observe that the sequence
$\{\la'_i-i+\hf\}_{i=1}^N$ is strictly decreasing, while
$\{-\la_i+i-\hf \}_{i=1}^N$ is strictly increasing. Also $A$ and $B$
are subsets of $C$. Since $\la'_i-i+\hf>0$ if and only if
$\la_i-i+\hf>0$, we have $i< j$ for all $\la'_i-i+\hf\in A$ and
$-\la_j+j-\hf\in B$. Furthermore, the sum of the cardinality of $A$
and the cardinality of $B$ equals the cardinality of $C$. So it is
enough to show $A\cap B=\emptyset$. Suppose that $\la'_i-i+\hf\in A$
and $-\la_j+j-\hf\in B$ with $\la'_i-i+\hf=-\la_j+j-\hf$. We have
$i< j$ and $\la'_i+\la_j=i+j-1$. If $\la'_i\ge j$, we have
$\la'_i+\la_j\ge j+i>j+i-1$. If $\la'_i< j$, we have $\la'_i+\la_j<
j+(i-1)=j+i-1$. In either case, $\la'_i+\la_j\not= i+j-1$. Thus we
have $A\cap B=\emptyset$, which completes the proof.
\end{proof}

\begin{lem}\label{finitecasimir} For $\mu\in \mc{X}_{m|\infty}$, we have
$(\mu+2\rho_c|\mu)_c=(\mu^\natural+2\rho_s|\mu^\natural)_s$.
\end{lem}

\begin{proof} A direct calculation shows that the lemma is
equivalent to the following identity for a partition
$\mu=(\mu_1,\mu_2,\cdots)$:
\begin{equation*}
\sum_{j>0}\mu_j^2-\sum_{j>0}2(j-1)\mu_j=\sum_{j>0}2j\mu_j'-\sum_{j>0}(\mu'_j)^2.
\end{equation*}
This identity is equivalent to
\begin{equation}\label{comb:id1}
\sum_{j=1}^N\left[\left(\mu_j-\left(j-\hf\right)\right)^2+\left(\mu'_j-\left(j-\hf\right)\right)^2\right]=2\sum_{j=1}^N\left(j-\hf\right)^2,
\end{equation}
where $N\ge\text{max}(\ell(\mu),\ell(\mu'))$. However
\eqnref{comb:id1} follows readily from \lemref{aux111} applied to
the partitions $\mu$ and $\mu'$.
\end{proof}

\begin{prop}\label{casimir:res} For
$\la\in\mc{P}_{m|\infty}$ and $\mu\in\overline{\h}^*$, the
irreducible $\SG_0$-module $L^0(\mu)$ is a component of
$\Lambda^k(\SG_{-1})\otimes {L}(\SG,\la^\natural)$ with $(\mu
+2\rho_s\vert\mu)_s=(\la^\natural+2\rho_s\vert\la^\natural)_s$ if
and only if there exists $w\in W^0_k$ with
$\mu=(w\circ\la)^\natural$. Furthermore, each such $L^0(\mu)$
appears with multiplicity one.
\end{prop}

\begin{proof}
Let $L^0(\mu)$ be an irreducible $\SG_0$-module in
$\Lambda^k(\SG_{-1})\otimes {L}(\SG,\la^\natural)$. By
\remref{decomposition}, we have $\mu=\eta^\natural$ for some
$\eta\in \mc{X}_{m|\infty}$. By \lemref{mult:same}, $L^0(\eta)$ is
an irreducible component of $\Lambda^k({\mathfrak g}_{-1})\otimes
{L}(\G,\la)$ with the same multiplicity.  By Lemma
\ref{finitecasimir}, if $(\mu
+2\rho_s\vert\mu)_s=(\la^\natural+2\rho_s\vert\la^\natural)_s$, then
we have $(\eta +2\rho_c\vert\eta)_c=(\la+2\rho_c\vert\la)_c$.
Furthermore by \propref{aux01}, $\eta = w\circ \la$ for some $w\in
W^0_k$, and the multiplicity of $L^0(\mu)$ is one.

Conversely, if $\mu=(w\circ\la)^\natural$ for some $w\in W^0_k$,
then by \lemref{finitecasimir} we get
\begin{equation*}
(\mu+2\rho_s\vert\mu)_s=(\la^\natural+2\rho_s\vert\la^\natural)_s.
\end{equation*}
By \propref{mult:same}, $L^0(w\circ\la)$ appears in
$\Lambda^k(\G_{-1})\otimes L(\G,\la)$ with multiplicity one. Hence
by \lemref{mult:same} $L^0(\mu)$ also appears in
$\Lambda^k(\SG_{-1})\otimes L(\SG,\la^\natural)$ with multiplicity
one.
\end{proof}

\section{Weak BGG-type resolutions for irreducible tensor
$\gl(m|n)$-modules}

Since $\SG/\ov{\mf{p}}$ is a $\ov{\mf{p}}$-module,
$D_k:=U(\SG)\otimes_{U(\overline{\mathfrak
p})}\Lambda^k(\SG/\overline{\mathfrak p})$ is a $\SG$-module with
$\SG$ acting on the first factor, for $k\ge 0$. Define the sequence
\begin{equation}\label{Koszul}
\cdots\stackrel{\partial_{k+1}}{\longrightarrow}D_k\stackrel{\partial_k}
{\longrightarrow}D_{k-1}\stackrel{\partial_{k-1}}{\longrightarrow}\cdots
\stackrel{\partial_1}{\longrightarrow}D_0\stackrel{\epsilon}{\longrightarrow}
\C\longrightarrow 0,
\end{equation}
where ${\epsilon}$ is the augmentation map from $U(\SG)$ to $\C$ and
$${\partial_k}(a\otimes
\bar{x}_1\bar{x}_2\cdots\bar{x}_k):=\sum_{j=1}^k ax_j\otimes
\bar{x}_1\cdots \widehat{\bar{x}}_j\cdots\bar{x}_k,$$ for $a\in
U(\SG)$ and $x_i\in \SG$. Here $\bar{x}_j$ denotes the image of
$x_j$ in $\SG/\overline{\mathfrak p}$ under the natural map. One
easily checks that the ${\partial_k}$'s are well-defined
$U(\SG)$-maps and \eqnref{Koszul} is a chain complex. The exactness
of \eqnref{Koszul} follows, for example, from the exactness of the
dual of the Koszul complex \cite[Appendix D.13]{Ku} (see also
\cite{KK}).

For $\la\in\mc{P}_{m|n}$ and $k\geq 0$, $Y_k:=D_k\otimes
L(\SG,\la^\natural)$ is a $\SG$-module. Tensoring \eqnref{Koszul}
with $L(\SG,\la^\natural)$ we obtain an exact sequence \cite{GL, Ku,
J}
\begin{equation}\label{standardres}
\cdots\stackrel{d_{k+1}}{\longrightarrow}Y_k\stackrel{d_k}{\longrightarrow}Y_{k-1}\stackrel{d_{k-1}}{\longrightarrow}\cdots
\stackrel{d_1}{\longrightarrow}Y_0\stackrel{d_0}{\longrightarrow}
L(\SG,\la^\natural)\longrightarrow 0,
\end{equation}
where $d_k:=\partial_k\otimes 1$ for $k>0$ and $d_0:=\epsilon\otimes
1$.

Let $V$ be a $\SG$-module, on which the action of $\SG_{+1}$ is
locally nilpotent. We define
\begin{equation*}
V^c:=\{\,v\in V\,\vert\, (\overline{\Omega}-c)^lv=0\ \  \text{for
$l\gg 0$}\,\},
\end{equation*}
i.e.~$V^c$ is the generalized $\overline{\Omega}$-eigenspace
corresponding to the eigenvalue $c\in \mathbb{C}$.  Clearly we
have $V=\bigoplus_{c\in\C}V^c$. Put
$$c_{\la}=(\la^\natural+2\rho_s|\la^\natural)_s.$$
The restriction of {\rm \eqnref{standardres}} to the generalized
$c_{\la}$-eigenspace of $\overline{\Omega}$ produces a resolution of
$\SG$-modules
\begin{equation}\label{resolution:verma}
\cdots\stackrel{d_{k+1}}{\longrightarrow}Y^{c_{\la}}_k\stackrel{d_k}
{\longrightarrow}Y^{c_{\la}}_{k-1}\stackrel{d_{k-1}}{\longrightarrow}\cdots
\stackrel{d_1}{\longrightarrow}Y^{c_{\la}}_0\stackrel{d_0}
{\longrightarrow}L(\SG,\la^\natural)\longrightarrow 0.
\end{equation}

\begin{prop}\label{finite:resolution} Assume that $n=\infty$. For $\la\in\mc{P}_{m|\infty}$,
we have a resolution of $\SG$-modules of the form
\begin{equation*}
\cdots\stackrel{d_{k+1}}{\longrightarrow}Z_k\stackrel{d_k}{\longrightarrow}Z_{k-1}\stackrel{d_{k-1}}{\longrightarrow}\cdots
\stackrel{d_1}{\longrightarrow}Z_0\stackrel{d_0}{\longrightarrow}L(\SG,\la^\natural)\longrightarrow
0
\end{equation*}
such that each $Z_k$ has a Kac flag.
Furthermore, $Z_k\cong\bigoplus_{w\in W^0_k} V(\SG,(w\circ
\la)^\natural)$ as $\SG_{-1}+\SG_0$-modules.
\end{prop}

\begin{proof} Observe that $Y_k\cong U(\SG)\otimes_{U(\ov{\mf{p}})}\left(\Lambda^k(\SG/\ov{{\mf
p}})\otimes L(\SG,\la^\natural)\right)$. Suppose that as
$\SG_0$-module, we have $\Lambda^k(\SG/\ov{{\mf p}})\otimes
L(\SG,\la^\natural)\cong\bigoplus_{\mu\in\mc I} L^0(\mu)$ for some
multiset of weights $\mc{I}$. The $\ov{\mf{p}}$-module
$\Lambda^k(\SG/\ov{{\mf p}})\otimes L(\SG,\la^\natural)$ has a
composition series, where the multiset of composition factors is
precisely the multiset of $\ov{\mf{p}}$-module $L^0(\mu)$,
$\mu\in\mc{I}$. Thus $Y_k$ has a Kac flag and $Y_k\cong
\bigoplus_{\mu\in\mc I}V(\SG,\mu)$ as $\SG_{-1}+\SG_0$-modules. Now
$\ov{\Omega}$ acts on $V(\SG,\mu)$ as the scalar
$(\mu+2\rho_s\vert\mu)_s$, and hence
$Z_k=Y_k^{c_\la}\cong\bigoplus_{\mu}V(\SG,\mu)$, where the summation
is over all $\mu\in\mc I$ such that
$(\mu+2\rho_s\vert\mu)_s=(\la^\natural+2\rho_s\vert\la^\natural)_s$.
\propref{casimir:res} now says that this set is precisely
$\{\,(w\circ\la)^\natural\,\vert\, w\in W^0_k\,\}$.
\end{proof}

\begin{cor}\label{cor:resolution} Assume that $n\in\N$. For
$\la\in\mc{P}_{m|n}$, we have a resolution of $\SG$-modules of the
form
\begin{align*}
\cdots\stackrel{d_{k+1}}{\longrightarrow}Z_{k,n}\stackrel{d_k}{\longrightarrow}Z_{k-1,n}\stackrel{d_{k-1}}{\longrightarrow}\cdots
\stackrel{d_1}{\longrightarrow}Z_{0,n}\stackrel{d_0}{\longrightarrow}L(\SG,\la^\natural)\longrightarrow
0
\end{align*}
such that each $Z_{k,n}$ has a Kac flag. Furthermore, $Z_{k,n}\cong
\bigoplus_{w\in W_{k}^0}V(\SG,(w\circ\la)^\natural)$ as
$\SG_{-1}+\SG_0$-modules. Here, by definition we have
$V(\SG,\nu^\natural)=0$ for $\nu\in \mc{X}_{m|\infty}$ with
$\nu_{1}>n$.
\end{cor}

\begin{proof}
The corollary follows from applying the truncation functor ${\mf{
tr}}_n$ \cite[Definition 4.4]{CWZ} upon the resolution in
\propref{finite:resolution} and using the facts that the truncation
functor is an exact functor and is compatible with both irreducible
and Kac modules \cite[Corollary 4.6]{CWZ}.
\end{proof}

\section{Strong BGG-type resolutions for irreducible tensor
$\gl(m|n)$-modules}\label{BGG:resolution}

For $n\in\N$ recall the definition of the super Bruhat ordering for
$\SG=\gl(m|n)$ on $\widetilde{\mc{X}}_{m|n}$ in \cite[\S 2-b]{B} ,
which we denote by $\preccurlyeq$. This gives a partial ordering on
$\widetilde{\mc{X}}_{m|n}$. We can restrict $\preccurlyeq$ to
$\mc{X}_{m|n}$, which can be defined for $\mc{X}_{m|\infty}$ as well
(cf.~\cite[Section 2.3]{CWZ}). Now we may also regard
$\widetilde{\mc{X}}_{m|n}$ as weights of $\G=\gl(m+n)$. In doing so
the usual Bruhat ordering of $\G$ determines a partial ordering
$\le$ on $\widetilde{\mc{X}}_{m|n}$ (see e.g.~\cite[Section
2.2]{CWZ}), which restricts to $\mc{X}_{m|n}$, and which in turn can
be defined for $\mc{X}_{m|\infty}$ as well. We have the following.

\begin{lem}\label{comp:bruhat} \cite[Lemma 6.6]{CWZ}
Let $\la,\mu\in{\mc{X}}_{m|\infty}$.  Then $\la\preccurlyeq\mu$ if
and only if $\la^\natural\le\mu^\natural$.
\end{lem}

In the remainder of this section we assume that $n\in\N$ unless
otherwise specified.

\begin{lem}\label{aux999} Let $n\in \N$ and $\la,\mu\in\widetilde{\mc{X}}_{m|n}$.
Suppose that $\mu\not\preccurlyeq\la$. Then
$${\rm Hom}_{\SG}(V(\SG,\mu),V(\SG,\la))=0.$$
\end{lem}

\begin{proof}
Suppose that ${\rm Hom}_{\SG}(V(\SG,\mu),V(\SG,\la))\not=0$. Then
$L(\SG,\mu)$ is a composition factor of the Kac module $V(\SG,\la)$.
It follows from  \cite[Corollary 3.36 (i)]{B} and \cite[Theorem
4.37]{B} that $\mu\preccurlyeq\la$.
\end{proof}

\begin{lem}\label{nohomomorphism}
Let $n\in \N$ and $\mu\in \widetilde{\mc{X}}_{m|n}$. Suppose that
$M$ is a finite-dimensional $\SG$-module with a Kac flag
\begin{align*}
0=M_0\subseteq M_1\subseteq M_2\subseteq\cdots\subseteq M_l=M,
\end{align*}
and ${\rm Hom}_{\SG}(M_i/M_{i-1},V(\SG,\mu))=0$ for all
$i=1,\cdots,l$. Then
$${\rm Hom}_{\SG}(M,V(\SG,\mu))=0.$$
\end{lem}

\begin{proof} Since $M$ is finite-dimensional we have
$M_i/M_{i-1}\cong V(\SG,\mu_i)$ with
$\mu_i\in\widetilde{\mc{X}}_{m|n}$ for all $i$. Consider the exact
sequence
$$0\rightarrow M_1\rightarrow M\rightarrow
M/M_1\rightarrow 0.$$ Noting that $M/M_1$ has a Kac flag of length
$l-1$, the lemma follows easily from the long exact sequence and
induction on $l$.
\end{proof}

\begin{lem}\label{noextension}
Let $n\in \N$ and $\la,\mu\in {\mc{X}}_{m|n}$. Suppose that $\la$
and $\mu$ are not comparable under the super Bruhat ordering. Then
$${\rm Ext}^1(V(\SG,\la),V(\SG,\mu))=0.$$
\end{lem}

\begin{proof}
Consider $P(\la)$ the projective cover (in the category of
finite-dimensional $\SG$-modules) of $L(\SG,\la)$. We have an exact
sequence
\begin{equation}\label{procover}
0\rightarrow K\rightarrow P(\la)\rightarrow V(\SG,\la)\rightarrow 0.
\end{equation}
Now $P(\la)$ has a Kac flag \cite[Proposition 2.5]{Z} and hence so
has $K$. By \cite[Theorem 4.37]{B}, $P(\la)$ is a tilting module and
if $V(\SG,\gamma)$ with $\gamma\not=\la$ appears in a Kac flag of
$P(\la)$, then $\gamma\in\widetilde{\mc{X}}_{m|n}$ and
$\gamma\succ\la$.

Now the induced long exact sequence from \eqnref{procover} gives
rise to the following exact sequence
$${\rm Hom}_{\SG}(K,V(\SG,\mu))\rightarrow {\rm Ext}^1(V(\SG,\la),V(\SG,\mu))\rightarrow 0.$$
Since all $V(\SG,\gamma)$ that appears in the Kac flag of $K$ are
such that $\gamma\succ\la$, we see that $\gamma\not\preccurlyeq\mu$
by hypothesis. Thus by Lemmas \ref{aux999} and \ref{nohomomorphism},
${\rm Hom}_{\SG}(K,V(\SG,\mu))=0$, and the lemma follows.
\end{proof}

\begin{thm}\label{true:BGG}
For $n\in \N$ and $\la\in\mc{P}_{m|n}$, we have a resolution of
$\SG$-modules of the form
\begin{align*}
\cdots\stackrel{d_{k+1}}{\longrightarrow}Z_{k,n}\stackrel{d_k}{\longrightarrow}Z_{k-1,n}\stackrel{d_{k-1}}{\longrightarrow}\cdots
\stackrel{d_1}{\longrightarrow}Z_{0,n}\stackrel{d_0}{\longrightarrow}L(\SG,\la^\natural)\rightarrow
0,
\end{align*}
where $Z_{k,n}\cong \bigoplus_{w\in
W_{k}^0}V(\SG,(w\circ\la)^\natural)$ as $\SG$-modules. As before, by
definition, we have $V(\SG,\nu^\natural)=0$ for $\nu\in
\mc{X}_{m|\infty}$ with $\nu_{1}>n$.
\end{thm}

\begin{proof} We have a natural embedding of
$\mc{X}_{m|N}\stackrel{\iota_{N,N+1}}{\longrightarrow}
\mc{X}_{m|N+1}$ for any $N\in\N$. Also we have the truncation map
$\mc{X}_{m|N+1}\stackrel{\textsf{Tr}_{N+1,N}}{\longrightarrow}\mc{X}_{m|N}$
\cite[Section 6.6]{CWZ} that sends an element
$\la=(\la_{-m},\cdots,\la_{N+1})$ to
$\la=(\la_{-m},\cdots,\la_{N})$, if $\la_{N+1}=0$, and to
$\emptyset$, otherwise. The usual Bruhat orderings of $\mc{X}_{m|N}$
and $\mc{X}_{m|N+1}$ are compatible in the following sense:
\begin{itemize}
\item[(i)] For $\la,\mu\in\mc{X}_{m|N}$, one has $\la\le\mu$ if and
only if $\iota_{N,N+1}(\la)\le\iota_{N,N+1}(\mu)$.
\item[(ii)] For $\la,\mu\in\mc{X}_{m|N+1}$ with
$\textsf{Tr}_{N+1,N}(\la)\not=\emptyset,\textsf{Tr}_{N+1,N}(\mu)\not=\emptyset$,
one has $\la\le\mu$ if and only if
$\textsf{Tr}_{N+1,N}(\la)\le\textsf{Tr}_{N+1,N}(\mu)$.
\end{itemize}
Thus the Bruhat ordering of $\mc{X}_{m|N}$ is compatible with that
of $\mc{X}_{m|\infty}$.

We view $\la$ as a weight of $\gl(m+\infty)$ and so as an element in
$\mc{X}_{m|\infty}$. For a fixed $j\in\N$, it is not hard to see
that the weights $\{\,w\circ\la\,\vert\, w\in W^0_j\,\}$ form a
finite set and they all may be regarded as lying in the same
$\mc{X}_{m|N}$, for $N\gg 0$. Thus we may regard them all as weights
of $\gl(m+N)$ for some $N\gg 0$. But for such weights, it is
well-known from classical theory of semi-simple Lie algebras that
they are not comparable under the usual Bruhat ordering (see
e.g.~\cite[Lemma 1.3.16]{Ku}). Thus viewing them as weights of
$\gl(m+\infty)$, they are not comparable under the Bruhat ordering,
either. Hence, by \lemref{comp:bruhat}, the weights
$(w\circ\la)^\natural$ are not comparable under the super Bruhat
ordering of $\gl(m|\infty)$.  The theorem now follows from a similar
compatibility of the super Bruhat orderings of $\gl(m\vert\infty)$
and of $\gl(m|n)$, \lemref{noextension} and \corref{cor:resolution}.
\end{proof}

\begin{rem}
Note that $W$ above is the infinite Weyl group of $\gl(m+\infty)$,
even though we are considering the finite-dimensional Lie
superalgebra $\gl(m|n)$.
\end{rem}

\begin{rem}
\thmref{true:BGG} has the counterpart in the case of $n=\infty$ as
well.
\end{rem}

Recall that for $\la,\mu\in\widetilde{\mc{X}}_{m|n}$ with
$\la\succcurlyeq\mu$ there is a {\em relative length function}
defined in \cite[\S3-g]{B}, which we denote by $\ov{\ell}(\mu,\la)$.
Fix $\la\in\mc{P}_{m|n}$ so that
$\la^\natural\in\widetilde{\mc{X}}_{m|n}$. For
$\mu\in\widetilde{\mc{X}}_{m|n}$ with $\la^\natural\succcurlyeq\mu$
define an {\em absolute length function} by
\begin{equation*}
\ov{\ell}(\mu):=\ov{\ell}(\mu,\la^\natural).
\end{equation*}
We can now formulate \thmref{true:BGG} intrinsically without
referring to the infinite Weyl group of $\gl(m+\infty)$ as follows.

\begin{thm}\label{reformtrue:BGG}
For $n\in \N$ and $\la\in\mc{P}_{m|n}$, we have a resolution of
$\SG$-modules of the form
\begin{align*}
\cdots\stackrel{d_{k+1}}{\longrightarrow}Z_{k,n}\stackrel{d_k}{\longrightarrow}Z_{k-1,n}\stackrel{d_{k-1}}{\longrightarrow}\cdots
\stackrel{d_1}{\longrightarrow}Z_{0,n}\stackrel{d_0}{\longrightarrow}L(\SG,\la^\natural)\rightarrow
0,
\end{align*}
where $Z_{k,n}\cong \bigoplus_{\ov{\ell}(\mu)=k }V(\SG,\mu)$ as
$\SG$-modules.
\end{thm}

\begin{proof}
For $\nu,\mu\in\widetilde{\mc{X}}_{m|n}$ recall Brundan's
Kazhdan-Lusztig polynomials $l_{\mu\nu}(q)$ of \cite[(2.18)]{B}. By
\cite[Theorem 4.51]{B} and \cite[Theorem 5.1]{Z} we have the
following cohomological interpretation:
\begin{equation*}
l_{\mu\nu}(-q^{-1})=\sum_{i=0}^\infty
\text{dim}\left[\text{Hom}_{\SG_0}\left(L^0(\mu),{\rm
H}^i\left(\SG_{+1};L(\SG,\nu)\right)\right)\right]q^i.
\end{equation*}
The calculation of the $\SG_{+1}$-cohomology groups in
\cite[Corollary 4.14]{CZ1} now implies that
\begin{equation*}
l_{\mu\la^\natural}(-q^{-1})=
\begin{cases}
q^k,\quad\text{if there exists }w\in W^0_k\text{ with
}\mu=(w\circ\la)^\natural\text{ and }(w\circ\la)_1\le n,\\
0,\quad\text{otherwise}.
\end{cases}
\end{equation*}
From \cite[Corollary 3.45]{B} we conclude that for such $\mu$ we
have $k=\ov{\ell}(\mu)$. On the other hand if
$\mu\in\widetilde{\mc{X}}_{m|n}$ with $\ov{\ell}(\mu)=k$, then
\cite[Corollary 3.45]{B} implies that
$l_{\mu\la^\natural}(-q^{-1})\not=0$ and hence $\mu$ is of the
form $(w\circ\la)^\natural$ with $w\in W^0_k$. Thus for
$\mu\in\widetilde{\mc{X}}_{m|n}$ the condition that there exists
$w\in W^0_k$ with $\mu=(w\circ\la)^\natural$ is equivalent to the
condition that $\ov{\ell}(\mu)=k$.  This together with
\thmref{true:BGG} completes the proof.
\end{proof}

We record the following corollary of the proof of
\thmref{reformtrue:BGG}.

\begin{cor}
Let $\la\in\mc{P}_{m|n}$.  As a $\SG_0$-module we have, for all
$k\in\ZZ_+$,
\begin{equation*}
{\rm
H}^k\left(\SG_{+1};L(\SG,\la^\natural)\right)\cong\bigoplus_{\ov{\ell}(\mu)=k}L^0(\mu).
\end{equation*}
\end{cor}

We conclude with an example, which shows that finite-dimensional
irreducible representations over a simple Lie superalgebra cannot be
resolved in terms of direct sums Verma modules in general.

\begin{example}\label{example} Let $\la\in\sh^*$ and let $\C_\la$
denote the one-dimensional $\sh$-module that transforms by $\la$. We
extend $\C_\la$ trivially to a $\ov{\mf{b}}$-module and denote by
$M(\SG,\la)={\rm Ind}_{\ov{\mf{b}}}^{\SG}\C_{\la}$ the Verma module
of highest weight $\la$. Suppose that $L(\SG,\la)$ can be resolved
in terms of Verma modules. Then we have an exact sequence of
$\SG$-modules of the form
\begin{align*}
\cdots{\longrightarrow}\bigoplus_{i\in I} M(\SG,\mu_i)
\stackrel{\psi}{\longrightarrow}M(\SG,\la)\stackrel{\phi}{\longrightarrow}
L(\SG,\la)\longrightarrow 0,
\end{align*}
and ${\rm Hom}_{\SG}(M(\SG,\mu_i),M(\SG,\la))\not=0$, for all $i\in
I$. It follows that there exist singular vectors $v_{i}$ of weight
$\mu_i$ in $M(\SG,\la)$. ${\rm Im}\psi={\rm Ker}\phi$ implies that
the unique maximal submodule of $M(\SG,\la)$ must be generated by
the proper singular vectors of $M(\SG,\la)$.

Now consider $\la=\delta_{-1}$ and $\SG=\mf{gl}(1|2)$ or
$\SG=\mf{sl}(1|2)$. In the sequel we will suppress $\SG$. One can
show by a direct calculation that the only proper singular vectors
in the Verma module $M(\delta_{-1})$ are scalar multiples of either
$\ov{E}_{2,1}v$ or $\ov{E}_{1,-1}\ov{E}_{2,1}v$, where $v$ is a
highest weight vector of $M(\delta_{-1})$. If $M_1$ is the submodule
of $M(\delta_{-1})$ generated by $\ov{E}_{2,1}v$, then $M_1$ is the
submodule generated by all proper singular vectors of
$M(\delta_{-1})$. But ${\rm dim}\left(M(\delta_{-1})/M_1\right)=4$
by the PBW Theorem and, since ${\rm dim} L(\delta_{-1})=3$, it
follows that $M_1$ cannot be the unique maximal submodule of
$M(\delta_{-1})$. Thus $L(\delta_{-1})$ cannot have a resolution in
terms of Verma modules. We note that $M(\delta_{-1})/M_1$ is
isomorphic to the Kac module of highest weight $\delta_{-1}$ and
$\ov{E}_{1,-1}\ov{E}_{2,-1}v$ is a singular vector in
$M(\delta_{-1})/M_1$.
\end{example}

\thmref{true:BGG},  the fact that ${\rm
H}^1\left(\SG_{+1};L(\SG,\la^\natural)\right)$ is irreducible, and
the discussion in \exref{example} imply the following.

\begin{cor}
Let $n\in \N$ and $\la\in\mc{P}_{m|n}$. The unique maximal
submodule of a reducible $V(\SG,\la^\natural)$ is generated by the
proper singular vector of $V(\SG,\la^\natural)$.
\end{cor}

\end{document}